\documentclass[11pt]{article}
\usepackage[letterpaper,twoside,outer=1.3in,vmargin=1.3in,]{geometry}
\usepackage{amsmath,amssymb,amsthm,amsfonts,enumerate,times,epsfig,color,hyperref,cleveref}
\usepackage{thmtools}
\usepackage{thm-restate}
\allowdisplaybreaks

\newcommand{\ignore}[1]{}

\newcommand{\1}{\mathbf{1}}
\newcommand{\0}{\mathbf{0}}

\newtheorem{theorem}{Theorem}%[section]

\newtheorem{PR}[theorem]{Proposition}
\newtheorem{CN}[theorem]{Conjecture}
\newtheorem{RE}[theorem]{Remark}

\newtheorem{DE}[theorem]{Definition}
\newtheorem{EG}[theorem]{Example}

\newcounter{claim_nb}[theorem]
\newtheorem{claim}[claim_nb]{Claim}
\newtheorem*{claim*}{Claim}
\newtheorem*{subclaim*}{Subclaim}

\newcounter{claim_nbs}[section]
\newcounter{subclaim_nb}[claim_nbs]
\newenvironment{cproof}
{\begin{proof}
 [Proof of Claim.]
 \vspace{-1.2\parsep}}
{\renewcommand{\qed}{\hfill $\Diamond$} \end{proof}}

\newif\ifnotes\notestrue
\ifnotes
\usepackage{color}
\newcommand{\notename}[2]{{\textcolor{red}{\footnotesize{\bf (#1:} {#2}{\bf ) }}}}

\newcommand{\anote}[1]{{\notename{Ahmad}{#1}}}
\newcommand{\ginote}[1]{{\notename{Giacomo}{#1}}}
\newcommand{\genote}[1]{{\notename{Gerard}{#1}}}

\else

\newcommand{\notename}[2]{{}}

\newcommand{\genote}[1]{}
\newcommand{\anote}[1]{}
\newcommand{\ginote}[1]{}

\fi

\title{Arc connectivity and submodular flows in digraphs}

\author{Ahmad Abdi \and G\'{e}rard Cornu\'{e}jols \and Giacomo Zambelli}
\begin{document}

\maketitle

\begin{abstract} Let $D=(V,A)$ be a digraph. For an integer $k\geq 1$, a \emph{$k$-arc-connected flip} is an arc subset of $D$ such that after reversing the arcs in it the digraph becomes (strongly) $k$-arc-connected.

The first main result of this paper introduces a sufficient condition for the existence of a $k$-arc-connected flip that is also a submodular flow for a crossing submodular function.
More specifically, given some integer $\tau\geq 1$, suppose $d_A^+(U)+(\frac{\tau}{k}-1)d_A^-(U)\geq \tau$ for all $U\subsetneq V, U\neq \emptyset$, where $d_A^+(U)$ and $d_A^-(U)$ denote the number of arcs in $A$ leaving and entering $U$, respectively. Let $\mathcal{C}$ be a crossing family over ground set $V$, and let $f:\mathcal{C}\to \mathbb{Z}$ be a crossing submodular function such that $f(U)\geq \frac{k}{\tau}(d_A^+(U)-d_A^-(U))$ for all $U\in \mathcal{C}$. Then $D$ has a $k$-arc-connected flip $J$ such that $f(U)\geq d_J^+(U)-d_J^-(U)$ for all $U\in \mathcal{C}$.
The result has several applications to Graph Orientations and Combinatorial Optimization. In particular, it strengthens Nash-Williams' so-called \emph{weak orientation theorem}, and proves a weaker variant of Woodall's conjecture on digraphs whose underlying undirected graph is $\tau$-edge-connected.

The second main result of this paper is even more general. It introduces a sufficient condition for the existence of capacitated integral solutions to the intersection of two submodular flow systems. This sufficient condition implies the classic result of Edmonds and Giles on the box-total dual integrality of a submodular flow system. It also has the consequence that in a weakly connected digraph, the intersection of two submodular flow systems is totally dual integral.\\

\noindent {\bf Keywords:} graph orientation, $k$-arc-connected flip, weak orientation theorem, Woodall's conjecture, submodular flows, total dual integrality
\end{abstract}

%%%%%%%%%%%%%%%%%%%%%%%%%%%%%%%%%%%%%%%
%%%%%%%%%%%%%%%%%%%%%%%%%%%%%%%%%%%%%%%
%%%%%%%%%%%%%%%%%%%%%%%%%%%%%%%%%%%%%%%
\section{Introduction}\label{sec:intro}
%%%%%%%%%%%%%%%%%%%%%%%%%%%%%%%%%%%%%%%
%%%%%%%%%%%%%%%%%%%%%%%%%%%%%%%%%%%%%%%
%%%%%%%%%%%%%%%%%%%%%%%%%%%%%%%%%%%%%%%

Graph Orientation is a rich area of Graph Theory. The basic problem consists in orienting the edges of an undirected graph in order to obtain a $k$-arc-connected digraph, and giving conditions under which such an orientation exists.
Various constraints on the orientation can be imposed, leading to an extensive literature in the area; for examples, see \cite{Frank80}, Chapter 9 of \cite{Frank11}, and Chapter 61 of \cite{Schrijver03}.
One can set up the basic problem equivalently in terms of digraphs, which is more appropriate for this paper: start from a digraph and flip the orientation of some of the arcs in order to obtain desired connectivity properties.

Let $D=(V,A)$ be a digraph, fixed throughout the rest of the introduction, unless stated otherwise. For $U\subseteq V$ denote by $\delta_D^+(U)$ and $\delta_D^-(U)$ the sets of arcs leaving and entering $U$, respectively. We shall drop the subscript $D$ whenever it is clear from the context. For $J\subseteq A$ and $U\subseteq V$, denote $d_J^+(U):=|\delta^+(U)\cap J|$ and $d_J^-(U):=|\delta^-(U)\cap J|$.

\begin{DE}\label{def:k-arc-connected-flip}
For an integer $k\geq 1$, a \emph{$k$-arc-connected flip of $D=(V,A)$} is a subset $J\subseteq A$ such that after flipping the arcs of $J$ the digraph becomes $k$-arc-connected, that is, $d_J^+(U)+d_A^-(U)-d_J^-(U)\geq k$ for all $U
\subsetneq V, U\neq \emptyset$, or equivalently, switching the roles of $U$ and $V\setminus U$, $d_J^+(U) -d_J^-(U)\leq d_A^+(U) - k$ for all $U \subsetneq V, U\neq \emptyset$.
\end{DE}

An important result is Nash-Williams' \emph{weak orientation theorem}, stating that there exists a $k$-arc-connected flip if, and only if, the underlying undirected graph of $D$ is $2k$-edge-connected (\cite{Nash-Williams69}, also see~\cite{Frank80,Ito23}).
Our main theorem strengthens (the nontrivial direction of) the weak orientation theorem in two ways. To state it we need to borrow a few notions from Submodular Optimization.

Let $\mathcal{C}$ be a family of subsets of $V$. Then $\mathcal{C}$ is a \emph{crossing family} over \emph{ground set} $V$ if, for all $U,W\in \mathcal{C}$ such that $U\cap W\neq \emptyset,U\cup W\neq V$, we have $U\cap W, U\cup W\in\mathcal{C}$. A function $f:\mathcal{C}\to \mathbb{R}$ is \emph{crossing submodular} over $\mathcal{C}$ if, for all $U,W\in \mathcal{C}$ such that $U\cap W\neq \emptyset, U\cup W\neq V$, we have  $f(U\cap W)+f(U\cup W)\leq f(U)+f(W)$. If we have $\geq$ or $=$ instead, then $f$ is a crossing \emph{supermodular} function or a crossing \emph{modular} function, respectively. For instance, $\mathcal{C}_1=\{U\subsetneq V:U\neq \emptyset\}$ is a crossing family, and $f_1:\mathcal{C}_1\to \mathbb{Z}$ defined as $f_1(U)=d_A^+(U)~\forall U\in \mathcal{C}_1$ is a crossing submodular function. %\footnote{$\forall$ is the ``for all" symbol.}
Another important example of a crossing family is $\mathcal{C}_2=\{U\subsetneq V: \delta_D^-(U)=\emptyset,U\neq \emptyset\}$, for which $f_2:\mathcal{C}_2\to \mathbb{Z}$ defined as $f_2(U)= d_A^+(U)~\forall U\in \mathcal{C}_2$ is a crossing modular function. More generally, given $w\in \mathbb{R}^A$, the function $f_3:2^V\to \mathbb{R}$ defined as $f(U)=w(\delta^+(U))-w(\delta^-(U))$ is crossing modular.

Let $\mathcal{C}$ be a crossing family over ground set $V$, and let $f:\mathcal{C}\to \mathbb{R}$ be a crossing submodular function. The linear system $y(\delta^+(U))-y(\delta^-(U))\leq f(U)~ \forall U\in \mathcal{C}$ is called a \emph{submodular flow system}, and every feasible solution is called a \emph{submodular flow}. Observe that for the crossing submodular function $f_4:\mathcal{C}_1\to \mathbb{Z}$ defined as $f_4(U)=f_1(U)-k~\forall C\in \mathcal{C}_1$, the incidence vector of a $k$-arc-connected flip of $D$ is a submodular flow, by \Cref{def:k-arc-connected-flip}.

\paragraph{Main result I.}
The following theorem introduces a sufficient condition for the existence of a $k$-arc-connected flip whose incidence vector is also a submodular flow for another crossing submodular function.

\begin{restatable}{theorem}{mainone}\label{main-1}
Let $\tau,k\geq 1$ be integers.
Let $D=(V,A)$ be a digraph where $d_A^+(U)+(\frac{\tau}{k}-1)d_A^-(U)\geq \tau$ for all $U\subsetneq V, U\neq \emptyset$. Let $\mathcal{C}$ be a crossing family over ground set $V$, and let $f:\mathcal{C}\to \mathbb{Z}$ be a crossing submodular function such that $f(U)\geq \frac{k}{\tau}(d_A^+(U)-d_A^-(U))$ for all $U\in \mathcal{C}$. Then $D$ has a $k$-arc-connected flip $J$ such that $f(U)\geq d_J^+(U)-d_J^-(U)$ for all $U\in \mathcal{C}$.
\end{restatable}

We shall prove \Cref{main-1} in \S\ref{sec:main-0-apps}. We shall discuss the complexity aspects following the theorem in \S\ref{sec:complexity}. As we will see, the inequalities on the cuts of $D$ above, can be verified in strongly polynomial time. Moreover, under some standard conditions on how $\mathcal{C}$ and $f$ are provided, $J$ can also be found in oracle strongly polynomial time.

Let us say a few more words about the inequalities on the cuts of $D$ in \Cref{main-1}. They are imposed for the simple reason that they are needed for our proof to work. That said, they possess some nice properties.

First, as a sanity check, we note that the inequalities readily imply that the underlying undirected graph is $2k$-edge-connected: for every $U\subsetneq V,U\neq \emptyset$, by adding the two inequalities corresponding to $U$ and $V\setminus U$, we get that $\frac{\tau}{k}(d^+_A(U)+d^-_A(U))\geq 2\tau$, implying in turn that $d^+_A(U)+d^-_A(U)\geq 2k$.

Secondly, we note that the inequalities are equivalent to asking that $\bar{y}=\frac{k}{\tau}\1$ satisfies $y(\delta^+(U)) -y(\delta^-(U))\geq k-d_A^-(U)$ for all $U\subsetneq V$, $U\neq \emptyset$. As $k$ increases, the inequalities become more strict. For $k=1$ the inequalities ask precisely that every \emph{dicut} (a cut with all arcs crossing in the same direction) has size at least $\tau$ (see \Cref{origin}), while for $k=\lfloor \tau/2\rfloor$ the inequalities \emph{almost} ask that every cut has size at least~$\tau$ (see \Cref{disjoint-dijoins}).

In \S\ref{sec:main-1-apps} we shall see applications of \Cref{main-1} to Graph Orientations and Combinatorial Optimization. For instance, we see that for $\tau=2k$ this result strengthens the weak orientation theorem and its near-Eulerian sharpening, and for $k=1$ it reduces to a recent result on decomposing $A$ into a \emph{dijoin} and a \emph{$(\tau-1)$-dijoin}~\cite{Abdi22+}. Other applications of \Cref{main-1} include an extension of the weak orientation theorem in a different direction than above, a weaker version of \emph{Woodall's conjecture} for digraphs with a $\tau$-edge-connected underlying undirected graph~\cite{Woodall78}, and a theorem on disjoint dijoins in $0,1$-weighted digraphs.

\paragraph{Main result II.} \Cref{main-1} is proved by utilizing a result on submodular flows. To explain it, we need a few notions from Integer Programming.
Let $A\in \mathbb{Q}^{m\times n}$ and $b\in \mathbb{Q}^m$. The linear system $Ax\leq b$ is \emph{totally dual integral (TDI)} if
for each $w\in \mathbb{Z}^n$, the dual of the linear program $\max\{w^\top x:Ax\leq b\}$ has an integral optimal solution whenever the LP admits an optimum~\cite{Edmonds77}. The system $Ax\leq b$ is \emph{box-TDI} if the system $Ax\leq b, \ell\leq x\leq u$ is TDI, for all $\ell,u\in \mathbb{Z}^n$ such that $\ell\leq u$.
An important result is that if $Ax\leq b$ is TDI and  $b\in \mathbb{Z}^m$, then $\{x:Ax\leq b\}$ is an \emph{integral} polyhedron, that is, every non-empty face of it contains an integral point~\cite{Hoffman74,Edmonds77}. In particular, if $Ax\leq b$ is box-TDI and $b\in \mathbb{Z}^m$, then $\{x:Ax\leq b\}$ is a \emph{box-integral} polyhedron, that is, $\{x:Ax\leq b,\ell\leq x\leq u\}=\{x:Ax\leq b\}\cap [\ell,u]$ is an integral polyhedron for all $\ell,u\in \mathbb{Z}^n$ such that $\ell\leq u$.

A classic result of Edmonds and Giles states that a submodular flow system is box-TDI~\cite{Edmonds77}. This important theorem laid the basis for numerous min-max theorems and polynomial and strongly polynomial algorithms for submodular flows. For in-depth surveys see \cite{Schrijver84,Frank88}, and for a more recent treatment we recommend Chapter 60 of \cite{Schrijver03} and Chapter 16 of \cite{Frank11}.

In contrast, given two crossing submodular functions $f_i:\mathcal{C}_i\to \mathbb{Z},i=1,2$ defined over (possibly different) crossing families $\mathcal{C}_i,i=1,2$ over the same ground set, the combined system
$y(\delta^+(U))-y(\delta^-(U))\leq f_i(U)~\forall~U\in \mathcal{C}_i,~i=1,2$, is not box-TDI and not even integral, as \Cref{bad-example} at the end of this section shows. Furthermore, finding an integral solution to the system includes NP-complete problems; see \S\ref{sec:two-subs} in the appendix for details.

Against this backdrop, \Cref{main-1} is significant since it provides a $0,1$ solution to the intersection of two submodular flow systems. The theorem is a consequence of the following result, which provides a sufficient condition for the existence of capacitated integral solutions to the intersection of two submodular flow systems.

\begin{theorem}\label{main-0}
Let $D=(V,A)$ be a digraph. For $i=1,2$, let $\mathcal{C}_i$ be a crossing family over ground set $V$, and let $f_i:\mathcal{C}_i\to \mathbb{Z}$ be a crossing submodular function, where $\min_{i=1,2}f_i(U)\leq 0$ for all $U\subsetneq V,U\neq \emptyset$ such that $\delta^+(U)=\delta^-(U)=\emptyset$.\footnote{We follow the convention that $f_i(U)=+\infty$ if $U\notin\mathcal{C}_i$.} Let
\begin{equation}\label{eq:sub_flow_intersection}
P:=\left\{y\in \mathbb{R}^A : y(\delta^+(U))-y(\delta^-(U))\leq f_i(U)~ \forall U\in \mathcal{C}_i,\, i=1,2\right\}.
\end{equation} Then the following statements hold: \begin{enumerate}[a.]
\item The system \eqref{eq:sub_flow_intersection} defining $P$ is TDI. In particular, $P$ is an integral polyhedron.
\item For every $\ell\in (\mathbb{Z}\cup\{-\infty\})^A$, $u\in (\mathbb{Z}\cup\{+\infty\})^A$ satisfying $\ell\leq u$ and the following \emph{cut condition}:
\begin{equation}\label{eq:cut-condition}
\min_{i=1,2} f_i(U)\leq u(\delta^+(U))-\ell(\delta^-(U))\quad \forall U\subsetneq V,\, U\neq \emptyset,
\end{equation} we have that every non-empty face of $P$ contains $y^\star\in \mathbb{Z}^A$ satisfying $\ell\leq y^\star\leq u$.
\end{enumerate}
\end{theorem}

We prove \Cref{main-0} in \S\ref{sec:main-0-proof}, and provide a delicate extension of it afterwards. We discuss the complexity aspects of the theorem in \S\ref{sec:complexity}. We do not know whether condition \eqref{eq:cut-condition} can be verified in polynomial time. However, as we shall see, under standard assumptions on how the crossing families $\mathcal{C}_i$ and the submodular functions $f_i$, $i=1,2$, are provided, there exists an oracle strongly polynomial time algorithm that returns  either a vector $y^\star$ as in the statement, or a subset $U$ violating the cut condition~\eqref{eq:cut-condition}.

It must be pointed out that, in contrast to the box-TDI-ness of submodular flow systems, the system \eqref{eq:sub_flow_intersection} is not box-TDI. In fact, $P\cap [\ell,u]$ is not necessarily integral even if $\ell,u$ satisfy the cut condition \eqref{eq:cut-condition}; for an example see \S\ref{sec:fractional} of the appendix.

Finally, let us say a few words about the first set of inequalities imposed on $\min_{i=1,2}f_i(U)$.
The condition that $\min_{i=1,2}f_i(U)\leq 0$ for all $U\subsetneq V,U\neq \emptyset$ such that $\delta^+(U)=\delta^-(U)=\emptyset$, is necessary for $P$ to be an integral polyhedron, as demonstrated by \Cref{bad-example} below. Note that if $P\neq \emptyset$, then the inequalities are equivalent to $\min_{i=1,2}f_i(U)= 0$ for all $U\subsetneq V,U\neq \emptyset$ such that $\delta^+(U)=\delta^-(U)=\emptyset$. Note further that these inequalities are implied by the cut condition \eqref{eq:cut-condition}.

\begin{EG}\label{bad-example}
Consider the digraph $D^\star$ with vertices $1,2,3,1',2',3'$ and arcs $a:=11',b:=22',c:=33'$. Let $\mathcal{C}^\star_1:=\{\{1,2\},\{1,2,3,1'\}\}$ and $\mathcal{C}^\star_2:=\{\{1,2\},\{1,2,3,2'\}\}$, which are clearly crossing families. Let $f^\star_1(\{1,2\})=f^\star_1(\{1,2,3,1'\})=1$ and $f^\star_2(\{1,2\})=f^\star_2(\{1,2,3,2'\})=1$, which yield integer-valued crossing submodular functions over $\mathcal{C}^\star_1,\mathcal{C}^\star_2$, respectively. Then the corresponding system \eqref{eq:sub_flow_intersection} is $y_a+y_b\leq 1, y_b+y_c\leq 1, y_c+y_a\leq 1$. This system is not integral, and therefore not box-TDI, as $(0.5,0.5,0.5)$ is a vertex of the polyhedron.
\end{EG}

\paragraph{Outline of the paper.} In \S\ref{sec:main-0-proof} we prove \Cref{main-0}, and provide an extension of it afterwards. In \S\ref{sec:main-0-apps} we present three applications of \Cref{main-0}: the original result of Edmonds and Giles, an application to digraphs with a connected underlying undirected graph, and of course \Cref{main-1}. In \S\ref{sec:main-1-apps} we present applications of \Cref{main-1}. In \S\ref{sec:complexity} we discuss the complexity aspects of the two main results. Finally, in \S\ref{sec:conclusion} we conclude with some open questions.

%%%%%%%%%%%%%%%%%%%%%%%%%%%%%%%%%%%%%%%
%%%%%%%%%%%%%%%%%%%%%%%%%%%%%%%%%%%%%%%
%%%%%%%%%%%%%%%%%%%%%%%%%%%%%%%%%%%%%%%
\section{Intersection of two submodular flow systems}\label{sec:main-0-proof}
%%%%%%%%%%%%%%%%%%%%%%%%%%%%%%%%%%%%%%%
%%%%%%%%%%%%%%%%%%%%%%%%%%%%%%%%%%%%%%%
%%%%%%%%%%%%%%%%%%%%%%%%%%%%%%%%%%%%%%%

In this section we prove \Cref{main-0}, for which we need two ingredients from Submodular Optimization and Network Flows. First we need the following result, essentially stating that the intersection of two base systems is box-TDI.

\begin{theorem}[see Theorem 49.8 of \cite{Schrijver03}, and \S14.4 of \cite{Frank11}]\label{box-TDI}
For $i=1,2$, let $\mathcal{C}_i$ be a crossing family over ground set $V$, let $f_i:\mathcal{C}_i\to \mathbb{Z}$ be a crossing submodular function, and let $k$ be an integer. Then the system $x(V)=k$; $x(U)\leq f_1(U)~\forall U\in \mathcal{C}_1$; $x(U)\leq f_2(U)~\forall U\in \mathcal{C}_2$ is box-TDI, and therefore defines a box-integral polyhedron.
\end{theorem}

Given a digraph $D=(V,A)$ and $b\in \mathbb{R}^V$, a \emph{$b$-transshipment} is a vector $y\in \mathbb{R}^A$ such that $y(\delta^+(v))-y(\delta^-(v))=b_v$ for every $v\in V$. Note that if there exists a $b$-transshipment, then $\1^\top b=0$ necessarily holds. We need the following result which characterizes the existence of capacity constrained integral $b$-transshipments.

\begin{theorem}[see Corollary 11.2f of \cite{Schrijver03}]\label{b-transshipment}
Let $D=(V,A)$ be a digraph. Let $b\in \mathbb{Z}^V$, $\ell\in (\mathbb{Z}\cup\{-\infty\})^A$, and $u\in (\mathbb{Z}\cup\{+\infty\})^A$ such that $\1^\top b=0$ and $\ell\leq u$. Then there exists a $b$-transshipment $y\in \mathbb{Z}^A$ such that $\ell\leq y\leq u$ if, and only if, $b(U)\leq u(\delta^+(U))-\ell(\delta^-(U))$ for all $U\subsetneq V, U\neq \emptyset$.
\end{theorem}

We are now ready to prove \Cref{main-0}.

\begin{proof}[Proof of \Cref{main-0}]
If $P=\emptyset$, then there is nothing to prove. Otherwise, $P\neq \emptyset$. Before we prove (a) and (b), let us set the scene. To this end, let $c\in\mathbb{Z}^A$ be a cost vector such that $\max\{c^\top y: y\in P\}$ admits an optimal solution, let $\omega^\star$ be the optimal value, and let $F$ be the optimal face. For $i=1,2$, let $\mathcal{D}_i$ be a subfamily of $\mathcal{C}_i$ such that $F=P\cap \{y: y(\delta^+(U))-y(\delta^-(U))= f_i(U)~\forall U\in \mathcal{D}_i,i=1,2\}$. Define the polyhedron $\widetilde{P}:=\{x\in \mathbb{R}^V:\1^\top x=0,x(U)\leq f_i(U)~\forall U\in \mathcal{C}_i,i=1,2\}$, and the face $\widetilde{F}:=\widetilde{P}\cap \{x: x(U)= f_i(U)~\forall U\in \mathcal{D}_i,i=1,2\}$.

\begin{claim} %claim 1
Let $x\in \mathbb{R}^V,y\in \mathbb{R}^A$ such that $y$ is an $x$-transshipment. Then $y\in P$ if and only if $x\in \widetilde{P}$; also $y\in F$ if and only if $x\in \widetilde{F}$.
\end{claim}
\begin{cproof}
Since an $x$-transshipment exists, $\1^\top x=0$ is automatically satisfied. The claim now follows from the equality $y(\delta^+(U))-y(\delta^-(U))=x(U)$ for all $U\subseteq V$.
\end{cproof}

Pick an arbitrary point $\bar{y}\in F$, and define $\bar x\in \mathbb{R}^V$ by $\bar x_v:=\bar y(\delta^+(v))-\bar y(\delta^-(v))~\forall v\in V$.
By Claim~1, $\bar x\in \widetilde{F}$, so $\widetilde{F}\neq \emptyset$. First we prove the second part as its proof is shorter and contains the crux of the argument. 

{\bf (b)} It suffices to find an integral point $y^\star\in F$ satisfying $\ell\leq y^\star\leq u$. By \Cref{box-TDI}, $\widetilde{P}$ is an integral polyhedron, hence $\widetilde{F}$ contains an integral point $b$.
Observe that $b(U)\leq f_i(U)$ for all $U\in \mathcal{C}_i,i=1,2$, so $b(U)\leq \min_{i=1,2} f_i(U)$ for all $U\subsetneq V,U\neq \emptyset$ (recall the convention that $f_i(U)=+\infty$ if $U\notin \mathcal{C}_i$). Thus, by the cut condition \eqref{eq:cut-condition}, we have that $b(U)\leq u(\delta^+(U))-\ell(\delta^-(U))$ for all $U\subsetneq V, U\neq \emptyset$. Thus, by \Cref{b-transshipment}, there exists a $b$-transshipment $y^\star\in \mathbb{Z}^A$ such that $\ell\leq y^\star\leq u$. Since $b\in \widetilde{F}$, it follows from Claim~1 that $y^\star\in F$. This is the desired point.

{\bf (a)}
To prove this part it suffices to show that the dual of $\max\{c^\top y: y\in P\}$ has an integral optimal solution. Let $M\in \{0,\pm 1\}^{V\times A}$ denote the node-arc incidence matrix of $D$. It is well-known that $M$ is a totally unimodular matrix (see Theorem 13.9 of \cite{Schrijver03}). Observe that $\bar{x}=M\bar{y}$.

\begin{claim}%claim 2
There exists $w\in \mathbb{Z}^V$ such that $w^\top M=c^\top$ and $w^\top \bar{x} = \omega^\star$.
\end{claim}
\begin{cproof}
Observe that if $y\in \mathbb{R}^A$ satisfies $My=\bar{x}$, then since $\bar{x}\in \widetilde{F}$, we have $y\in F$ by Claim~1, so $c^\top y=\omega^\star$ by definition. Thus, the linear system $My=\bar{x}$ where $y\in \mathbb{R}^A$ is a vector of variables, implies the equation $c^\top y=\omega^\star$. Subsequently, there exists $w\in \mathbb{R}^V$ such that $w^\top M=c^\top$ and $w^\top \bar{x}=\omega^\star$. Since $M$ is totally unimodular, and $c\in \mathbb{Z}^A$, we may choose $w\in \mathbb{Z}^V$ such that $w^\top M= c^\top$. Note that $w^\top \bar{x} = w^\top M \bar{y} = c^\top \bar{y} = \omega^\star$.
\end{cproof}

\begin{claim} %claim 3
$\max\{w^\top x:x\in \widetilde{P}\}=\omega^\star$.
\end{claim}
\begin{cproof}
$(\geq)$ follows from $w^\top \bar{x}=\omega^\star$.
$(\leq)$
Let $x'\in \widetilde{P}$. We know that $x'(U)\leq f_i(U)$ for all $U\in \mathcal{C}_i,i=1,2$, so $x'(U)\leq \min_{i=1,2}f_i(U)$ for all $U\subsetneq V,U\neq \emptyset$. By hypothesis, the right-hand side is at most $0$ for all $U\subsetneq V,U\neq \emptyset$ such that $\delta^+(U)=\delta^-(U)=\emptyset$, so for all such $U$, $x'(U)\leq 0$. Thus, by \Cref{b-transshipment} with the choices of $u=+\infty$ and $\ell=-\infty$, there exists an $x'$-transshipment $y'\in \mathbb{R}^A$, i.e.\ $My'=x'$. Observe that $y'\in P$ by Claim~1. Thus, $
w^\top x'=w^\top M y'=c^\top y'\leq
\omega^\star$, where the last inequality follows from the definition of $\omega^\star$ and the fact that $y'\in P$.
\end{cproof}

Now consider the dual of
 $\max\{w^\top x:x\in \widetilde{P}\}$:
\begin{equation}\label{eq:submodular-dual}
\begin{array}{crrl}
\min &
\sum_{i=1,2}\sum_{U\in\mathcal{C}_i} f_i(U) z^i_U
&&\\ \\
\text{s.t.} &
\sum_{i=1,2}\sum_{U\in\mathcal{C}_i} \chi^U z^i_U+ \1 \mu
&=& w \\ \\
&z^i_U&\geq &0 \quad  U\in \mathcal{C}_i,\, i=1,2,
\end{array}
\end{equation}
where $\mu\in \mathbb{R}$ is the dual variable corresponding to $\1^\top x=0$, and $\chi^U$ is the incidence vector of $U$ as a subset of $V$. By \Cref{box-TDI}, the system of constraints of $\widetilde{P}$ is TDI. Thus, since $w$ is integral, it follows that \eqref{eq:submodular-dual} has an integral optimal solution $(\bar z,\bar \mu)$. By Claim~3 and LP Strong Duality, \eqref{eq:submodular-dual} has optimal value $\omega^\star$. Thus, $\sum_{i=1,2}\sum_{U\in\mathcal{C}_i} f_i(U) \bar{z}^i_U=\omega^\star$.

\begin{claim} %claim 4
$\bar z=(\bar z_U^i)_{U\in \mathcal{C}_i,\, i=1,2}$ is an optimal solution to the dual of $\max\{c^\top y: y\in P\}$:
\begin{equation}\label{eq:submodular-flow-intersection-dual}
\begin{array}{crrl}
\min &
\sum_{i=1,2}\sum_{U\in\mathcal{C}_i} f_i(U) z^i_U
&&\\ \\
\text{s.t.} &
\sum_{i=1,2}\sum_{U\in\mathcal{C}_i} \left(\chi^{\delta^+(U)}-\chi^{\delta^-(U)}\right)z^i_U
&=& c \\ \\
&z^i_U&\geq &0 \quad  U\in \mathcal{C}_i,\, i=1,2,
\end{array}
\end{equation} where $\chi^{\delta^+(U)},\chi^{\delta^-(U)}$ are the incidence vectors of $\delta^+(U),\delta^-(U)$ as subsets of $A$.
\end{claim}
\begin{cproof}
By definition, $\omega^\star=\max\{c^\top y: y\in P\}$, so by LP Strong Duality, it suffices to prove that $\bar{z}$ is a feasible solution to \eqref{eq:submodular-flow-intersection-dual} with objective value $\omega^\star$. The latter is indeed the case as we argued above. Let us prove feasibility. Clearly, $\bar{z}\geq \0$. Moreover, $$\sum_{i=1,2}\sum_{U\in\mathcal{C}_i} \left(\chi^{\delta^+(U)}-\chi^{\delta^-(U)}\right)\bar z^i_U =M^\top \left(\sum_{i=1,2}\sum_{U\in\mathcal{C}_i} \chi^U \bar z^i_U+ \1 \bar \mu\right)= M^\top w  = c$$
where the first equality follows from $M^\top \chi^U=\chi^{\delta^+(U)}-\chi^{\delta^-(U)}$ for every $U\subsetneq V$, $U\neq \emptyset$, and $M^\top \1=0$, the second equality from the feasibility of $(\bar{z},\bar{\mu})$ for \eqref{eq:submodular-dual}, and the third equality from the definition of $w$. Thus, $\bar{z}$ is feasible for \eqref{eq:submodular-flow-intersection-dual}, as required.
\end{cproof}

Claim~4 finishes the proof of the first part.
\end{proof}

\paragraph{An extension.} By a delicate analysis of the proof, we can show that \Cref{main-0} admits the following extension. Below, for a vector $x\in \mathbb{R}^n$, $x^+$ and $x^-$ denote the vectors in $\mathbb{R}^n$ defined as $x^+_i=\max\{x_i,0\}$ and $x^-_i=\min\{x_i,0\}$ for all $i\in [n]$.

\begin{theorem}\label{general-main-0}
Let $V$ be a finite set, and let $M$ be a $|V|$-by-$m$ totally unimodular matrix with rows indexed by the elements of $V$, such that $M^\top \1=\0$. For $i=1,2$, let $\mathcal{C}_i$ be a crossing family over ground set $V$, and let $f_i:\mathcal{C}_i\to \mathbb{Z}$ be a crossing submodular function, where $\min_{i=1,2}f_i(U)\leq 0$ for all $U\subsetneq V,U\neq \emptyset$ such that $M^\top \chi^U=\0$. Let
\begin{equation}\label{eq:general-sub_flow_intersection}
P:=\left\{y\in \mathbb{R}^m : (\chi^U)^\top M y\leq f_i(U)~ \forall U\in \mathcal{C}_i,\, i=1,2\right\}.
\end{equation} Then the following statements hold: \begin{enumerate}[a.]
\item The system \eqref{eq:general-sub_flow_intersection} defining $P$ is TDI. In particular, $P$ is an integral polyhedron.
\item For every $\ell\in (\mathbb{Z}\cup\{-\infty\})^m$, $u\in (\mathbb{Z}\cup\{+\infty\})^m$ satisfying $\ell\leq u$ and the following condition:
\begin{equation}\label{eq:general-cut-condition}
\min_{i=1,2} f_i(U)\leq u^\top (M^\top \chi^U)^+  -\ell^\top (M^\top \chi^U)^-\quad \forall U\subsetneq V,\, U\neq \emptyset,
\end{equation} we have that every non-empty face of $P$ contains $y^\star\in \mathbb{Z}^A$ satisfying $\ell\leq y^\star\leq u$.
\end{enumerate}
\end{theorem}
\begin{proof}[Proof sketch.]
The proof is almost identical to that of \Cref{main-0} but with a few modifications which we highlight as follows. First, the optimal face $F$ is now defined as $F=P\cap \{y: (\chi^U)^\top My= f_i(U)~\forall U\in \mathcal{D}_i,i=1,2\}$. Furthermore, $\widetilde{P}$ and its face $\widetilde{F}$ are defined as before, namely $\widetilde{P}:=\{x\in \mathbb{R}^V:\1^\top x=0,x(U)\leq f_i(U)~\forall U\in \mathcal{C}_i,i=1,2\}$, and the face $\widetilde{F}:=\widetilde{P}\cap \{x: x(U)= f_i(U)~\forall U\in \mathcal{D}_i,i=1,2\}$.

The statement of Claim~1 needs to be modified as follows:

\begin{claim} 
Let $x\in \mathbb{R}^V,y\in \mathbb{R}^m$ such that $x=My$. Then $y\in P$ if and only if $x\in \widetilde{P}$; also $y\in F$ if and only if $x\in \widetilde{F}$.
\end{claim}
\begin{cproof}
The claim follows from the definitions of $\widetilde{P}$ and $\widetilde{F}$, and the fact that $\1^\top x=\1^\top My=0$ since $M^\top \1=\0$. 
\end{cproof}

Furthermore, we need the following additional claim.

\begin{claim}
Let $b\in \mathbb{R}^V$ such that $\1^\top b=0$.
\begin{enumerate}[i.]
  \item The system $My=b$ is feasible if and only if $b(U)\le 0$ for all $U\subsetneq V,U\neq \emptyset$ such that $M^\top \chi^U=\0$.
  \item The system $My=b$, $\ell\le y\le u$ is feasible if and only if $b(U)\le u^\top (M^\top \chi^U)^+  -\ell^\top (M^\top \chi^U)^-$ for all $U\subsetneq V,U\neq \emptyset$.
\end{enumerate}
\end{claim}
\begin{cproof}
The ``only if'' direction of both statements is trivial, so we focus on the ``if'' statements. 

For part i), $My=b$ is feasible if $b^\top \bar{z}=0$ for all $\bar{z}\in \{z\in\mathbb{R}^V:M^\top z=\0\}$. By shifting $\bar{z}$ by some $\alpha\1,\alpha\in \mathbb{R}$, if necessary, we may assume that $\bar{z}\geq \0$. (Note that $M^\top z=M^\top (z+\alpha\1)$ and $b^\top z=b^\top (z+\alpha\1)=\0$, because $M^\top \1=\0$ and $b^\top \1=0$.) Furthermore, by scaling down $\bar{z}$, if necessary, we may assume that $\bar{z}\leq \1$. Thus, we can only focus on the points $\bar{z}$ in the polyhedron $Q:=\{z:M^\top z=\0,\, \0\le z\le \1\}$. In fact, by basic polyhedral theory, it suffices to focus on the extreme points $\bar{z}$ of $Q$. Since $M$ is totally unimodular, $\bar{z}$ must be integral, so $\bar{z}=\chi^U$ for some $U\subseteq V$. If $U\in \{\emptyset,V\}$, then clearly $b^\top \bar{z}=0$, so we are done. Otherwise, since $M^\top \chi^U=\0$, it follows from the assumption that $b(U)\leq 0$. Since $M^\top \1=\0$, we also have $M^\top \chi^{V\setminus U}=\0$, so again by assumption $b(V\setminus U)\leq 0$. Thus, $b(U),b(V\setminus U)\le 0$, and since $\1^\top b=0$, we get that $b^\top \bar{z} = b(U)=0$, as required.

For part ii), by Farkas' lemma, the system $My=b$, $\ell\le y\le u$ is feasible if, for all $(z,v,w)\in\mathbb{R}^V\times \mathbb{R}^m\times \mathbb{R}^m$ such that $v,w\ge \0$ and $M^\top z=v-w$, it follows that $b^\top z\le u^\top v-\ell^\top w$. Once again, by shifting $z$ by some $\alpha\1,\alpha\in \mathbb{R}$, if necessary, we may assume that $z\ge \0$. Furthermore, by scaling down $(z,v,w)$, if necessary, we may assume that $z\le \1$. As before, since $M$ is totally unimodular, we can focus on $z=\chi^U$ for some $U\subseteq V$. Since $u\ge \ell$, the choice of $v,w\geq \0$ such that $M^\top \chi^U=v-w$ that minimizes $u^\top v-\ell^\top w$, is $v=(M^\top \chi^U)^+$ and $w=(M^\top \chi^U)^-$. The statement follows.
\end{cproof}

Claim~2~(i) ensures that if $x\in \widetilde{P}$, then $My=x$ is feasible, because $x(U)\leq \min_{i=1,2}f_i(U)\leq 0$ for all $U\subsetneq V,U\neq \emptyset$ such that $M^\top \chi^U=\0$. Furthermore, Claim~(2)~(ii) ensures that if $\ell\le u$ satisfies the condition \eqref{eq:general-cut-condition}, then for all $x\in \widetilde{P}$, the system $My=x$, $\ell\le y\le u$ is feasible. With these observations, the rest of the proof of the theorem follows exactly the proof of \Cref{main-0}.
\end{proof}

%}

%%%%%%%%%%%%%%%%%%%%%%%%%%%%%%%%%%%%%%%
%%%%%%%%%%%%%%%%%%%%%%%%%%%%%%%%%%%%%%%
%%%%%%%%%%%%%%%%%%%%%%%%%%%%%%%%%%%%%%%
\section{Applications of \Cref{main-0}, and proof of \Cref{main-1}}\label{sec:main-0-apps}
%%%%%%%%%%%%%%%%%%%%%%%%%%%%%%%%%%%%%%%
%%%%%%%%%%%%%%%%%%%%%%%%%%%%%%%%%%%%%%%
%%%%%%%%%%%%%%%%%%%%%%%%%%%%%%%%%%%%%%%

In this section we discuss three applications of \Cref{main-0}, one of which is \Cref{main-1}.

\paragraph{First application.} \Cref{main-0} implies the classic theorem of Edmonds and Giles~\cite{Edmonds77}. One way to prove it is to use a recent characterization of box-TDI systems. Consider a polyhedron $Q:=\{y:Ay\leq b\}$. For an integer $k\geq 1$, the $k\textsuperscript{th}$ dilation of $Q$ is $kQ:=\{y:Ay\leq k b\}$. $Q$ is \emph{principally box-integral} if $kQ$ is box-integral for all integers $k\geq 1$ such that $kQ$ is integral. This notion was coined recently by Chervet, Grappe, and Robert who proved that $Ax\leq b$ is box-TDI if, and only if, $Ax\leq b$ is TDI and $Q$ is principally box-integral~\cite{Chervet21}.

\begin{theorem}[\cite{Edmonds77}]
Let $D=(V,A)$ be a digraph, let $\mathcal{C}$ be a crossing family over ground set $V$, and let $f:\mathcal{C}\to \mathbb{Z}$ be a crossing submodular function. Then $y(\delta^+(U))-y(\delta^-(U))\leq f(U) ~ \forall U\in \mathcal{C}$ is box-TDI.
\end{theorem}
\begin{proof} We need the following two claims.

\begin{claim} %claim 1
$y(\delta^+(U))-y(\delta^-(U))\leq f(U) ~ \forall U\in \mathcal{C}$ is TDI.
\end{claim}
\begin{cproof}
Let $\mathcal{C}_1:=\mathcal{C}$ and $f_1:=f$. Let $\mathcal{C}_2:=\{U\subsetneq V:U\neq \emptyset,\delta^+(U)=\delta^-(U)=\emptyset\}$ and $f_2(U):=0$ for all $U\in \mathcal{C}_2$. Then $\mathcal{C}_2$ is a crossing family and $f_2$ is a crossing submodular function defined over $\mathcal{C}_2$. Moreover, $\min_{i=1,2}f_i(U)\leq 0$ for all $U\subsetneq V,U\neq \emptyset$ such that $\delta^+(U)=\delta^-(U)=\emptyset$. Subsequently, it follows from \Cref{main-0}~(a) that the combined system $y(\delta^+(U))-y(\delta^-(U))\leq f_i(U) ~ \forall U\in \mathcal{C}_i,i=1,2$ is TDI. However, this system is precisely $y(\delta^+(U))-y(\delta^-(U))\leq f(U) ~ \forall U\in \mathcal{C}$, so the claim follows.
\end{cproof}

Let $Q:=\{y\in \mathbb{R}^A : y(\delta^+(U))-y(\delta^-(U))\leq f(U) ~ \forall U\in \mathcal{C}\}$.

\begin{claim} %claim 2
For every integer $k\geq 1$, $kQ$ is box-integral.
\end{claim}
\begin{cproof}
It suffices to prove this for $k=1$, given that $kf$ is a crossing submodular function for every integer $k\geq 1$. Take $\ell,u\in \mathbb{Z}^A$ such that $\ell\leq u$ and consider a nonempty face of $Q\cap [\ell,u]$, say $F=Q\cap \{y\in \mathbb{R}^A : y(\delta^+(U))-y(\delta^-(U))= f(U)~ \forall U\in \mathcal{D},\, y_e=\ell_e ~ \forall e\in A_\ell,\, y_e=u_e ~ \forall e\in A_u\}$, where $\mathcal{D}\subseteq \mathcal{C}$; $A_\ell,A_u\subseteq A$; $A_\ell\cap A_u=\emptyset$. We need to show that $F$ contains an integer point.
To this end, let $\mathcal{C}_1:=\mathcal{C}$ and for every $U\in \mathcal{C}_1$ define $$f_1(U):=f(U)+\ell(\delta^-(U)\cap A_\ell)-\ell(\delta^+(U)\cap A_\ell)+u(\delta^-(U)\cap A_u)-u(\delta^+(U)\cap A_u).$$
Let $A':=A- (A_\ell\cup A_u)$, $D':=(V,A')$,
$\mathcal{C}_2:=\{U\subsetneq V:U\neq \emptyset\}$, and $f_2(U):=u(\delta_{D'}^+(U))-\ell(\delta_{D'}^-(U))$ for all $U\in \mathcal{C}_2$.
Observe that $f_1$ is a crossing submodular function, because it is the sum of a crossing submodular function $f$, and crossing modular functions $\ell(\delta^-(U)\cap A_\ell)-\ell(\delta^+(U)\cap A_\ell)$ and $u(\delta^-(U)\cap A_u)-u(\delta^+(U)\cap A_u)$.
Observe also that $f_2(U)$, as the sum of a crossing modular function $u(\delta_{D'}^+(U))-u(\delta_{D'}^-(U))$ and a crossing submodular function $(u-\ell)(\delta_{D'}^-(U))$, is a crossing submodular function.

Consider the polyhedron $Q':=\{y\in \mathbb{R}^{A'} : y(\delta_{D'}^+(U))-y(\delta_{D'}^-(U))\leq f_i(U) ~ \forall U\in \mathcal{C}_i,\, i=1,2\}$ and its face $F':=Q'\cap \{y\in \mathbb{R}^{A'} : y(\delta_{D'}^+(U))-y(\delta_{D'}^-(U))= f_1(U) ~ \forall U\in \mathcal{D}\}$. Note that $F'$ is nonempty since it is the restriction of $F$ to $\mathbb{R}^{A'}$. We shall apply \Cref{main-0}~(b) to $Q'$.
Observe that $\min_{i=1,2}f_i(U)\leq f_2(U)=u(\delta_{D'}^+(U))-\ell(\delta_{D'}^-(U))$ for all $U\in \mathcal{C}_2=\{U\subsetneq V:U\neq \emptyset\}$, so the cut condition \eqref{eq:cut-condition} is satisfied for $\ell,u$. In particular, $\min_{i=1,2}f_i(U)\leq 0$ for all $U\subsetneq V,U\neq \emptyset$ such that $\delta_{D'}^+(U)=\delta_{D'}^-(U)=\emptyset$. Hence, by \Cref{main-0}~(b), $F'$ contains an integer point $y^\star\in\mathbb{R}^{A'}$ such that $\ell_e\le y^\star_e\le u_e$ for all $e\in A'$. Extend the point $y^\star$ to $\mathbb{R}^A$ by defining $y^\star_e:=\ell_e$ for all $e\in A_\ell$, and $y^\star_e:=u_e$ for all $e\in A_u$. Then, by the definition of $f_1$, we have $y^\star\in F\cap \mathbb{Z}^A$, as desired.
\end{cproof}

It follows from Claim~2 that $Q$ is principally box-integral. This, together with Claim~1 and the theorem of Chervet, Grappe, and Robert~\cite{Chervet21}, implies that $y(\delta^+(U))-y(\delta^-(U))\leq f(U) ~ \forall U\in \mathcal{C}$ is box-TDI.
\end{proof}

\paragraph{Second application.}
A digraph is \emph{weakly connected} if its underlying undirected graph is connected. The next application is the following result, which surprisingly seems to be new. Observe that the weak connectivity assumption cannot be dropped, as shown by \Cref{bad-example}.

\begin{theorem}\label{thm:integral_intersection_weakly_connected}
Let $D=(V,A)$ be a weakly connected digraph and, for $i=1,2$, let $\mathcal{C}_i$ be a crossing family over ground set $V$ and $f_i:\mathcal{C}_i\to \mathbb{Z}$ be a crossing submodular function. Then the system in \eqref{eq:sub_flow_intersection} is TDI, and in particular, the polyhedron $P$ is integral.
\end{theorem}
\begin{proof}
Since $D$ is weakly connected, the condition $\min_{i=1,2}f_i(U)\leq 0$ for all $U\subsetneq V,U\neq \emptyset$ is vacuously true, because there is no such $U$. Thus, the result follows from \Cref{main-0}~(a).
\end{proof}

\paragraph{Third application.} The final application is \Cref{main-1}, which we restate for convenience.

\mainone*
\begin{proof}
Let $\mathcal{C}_1:=\mathcal{C}$, $f_1:=f$, and define $\mathcal{C}_2:=\{U\subsetneq V: U\neq \emptyset\}$,  $f_2(U):=d_A^+(U)-k$ for all $U\in\mathcal{C}_2$. Observe that $f_2$ is a crossing submodular function. Consider the vector $y\in \mathbb{R}^A$ that assigns $\frac{k}{\tau}$ to every arc $a\in A$.
Then $y(\delta^+(U))-y(\delta^-(U))\leq f_1(U)$ for all $U\in \mathcal{C}_1$, by one of our assumptions. Moreover, for all $U\subsetneq V,\, U\neq \emptyset$, our assumption implies that
$d_A^+(V\setminus U)+(\frac{\tau}{k}-1)d_A^-(V\setminus U)\geq \tau$, which in turn can be written as $y(\delta^+(V\setminus U))-y(\delta^-(V\setminus U))\geq k- d_A^-(V\setminus U)$, which is equivalent to $y(\delta^+(U))-y(\delta^-(U))\leq f_2(U)$. Furthermore, $f_2(U)\leq d_A^+(U)$ for all $U\subsetneq V$, $U\neq \emptyset$, so the cut condition \eqref{eq:cut-condition} holds for the choices of $\ell=\0,u=\1$. (Observe further that $\min_{i=1,2}f_i(U)\leq 0$ for all $U\subsetneq V$, $U\neq \emptyset$ such that $\delta^+(U)=\delta^-(U)=\emptyset$, holds vacuously as there is no such $U$.) It therefore follows from \Cref{main-0}~(b) that there exists $y^\star\in \{0,1\}^A$ such that $y^\star(\delta^+(U))-y^\star(\delta^-(U))\leq f_i(U)$ for all $U\in \mathcal{C}_i$, $i=1,2$. Let $J:=\{a\in A:y^\star_a=1\}$. Then
$d_J^+(U)-d_J^-(U)
=y^\star(\delta^+(U))-y^\star(\delta^-(U))
\leq f(U)$ for all $U\in \mathcal{C}$. Moreover, $
d_J^+(U)-d_J^-(U)
\leq f_2(U)= d_A^+(U)-k$ for all $U\subsetneq V,U\neq \emptyset$,
implying in turn that $J$ is a $k$-arc-connected flip. Thus, $J$ is the desired set.
\end{proof}

%%%%%%%%%%%%%%%%%%%%%%%%%%%%%%%%%%%%%%%
%%%%%%%%%%%%%%%%%%%%%%%%%%%%%%%%%%%%%%%
%%%%%%%%%%%%%%%%%%%%%%%%%%%%%%%%%%%%%%%
\section{Applications of \Cref{main-1}}\label{sec:main-1-apps}
%%%%%%%%%%%%%%%%%%%%%%%%%%%%%%%%%%%%%%%
%%%%%%%%%%%%%%%%%%%%%%%%%%%%%%%%%%%%%%%
%%%%%%%%%%%%%%%%%%%%%%%%%%%%%%%%%%%%%%%

We present several applications of \Cref{main-1} to Graph Orientations and Combinatorial Optimization.

%----------------------------------------------------------------------->
\subsection{An extension of the weak orientation theorem}
%----------------------------------------------------------------------->

For $\tau=2k$, \Cref{main-1} gives the following strengthening of the weak orientation theorem (that every digraph whose underlying undirected graph is $2k$-edge-connected, has a $k$-arc-connected flip).

\begin{theorem}\label{NW}
Let $D=(V,A)$ be a digraph whose underlying undirected graph is $2k$-edge-connected. Let $\mathcal{C}$ be a crossing family over ground set $V$, and let $f:\mathcal{C}\to \mathbb{Z}$ be a crossing submodular function such that $f(U)\geq \frac{1}{2}(d_A^+(U)-d_A^-(U))$ for all $U\in \mathcal{C}$. Then $D$ has a $k$-arc-connected flip $J$ such that $f(U)\geq d_J^+(U)-d_J^-(U)$ for all $U\in \mathcal{C}$.\qed
\end{theorem}

A digraph is \emph{near-Eulerian} if at every vertex the in-degree and out-degree differ by at most one. \Cref{NW} implies the following well-known extension of the weak orientation theorem.

\begin{theorem}[\cite{Nash-Williams69}]
Let $D=(V,A)$ be a digraph whose underlying undirected graph is $2k$-edge-connected. Then there exists a $k$-arc-connected flip $J$ such that after flipping its arcs the digraph becomes near-Eulerian.
\end{theorem}
\begin{proof}
Let $\mathcal{C}:=\{\{u\},V\setminus u:u\in V\}$, and $f(U):=\left\lceil \frac{1}{2} (d_A^+(U)-d_A^-(U))\right\rceil$ for all $U\in \mathcal{C}$. Clearly, $\mathcal{C}$ is a crossing family, $f$ is a crossing submodular function over $\mathcal{C}$, and $f(U)\geq \frac{1}{2}(d_A^+(U)-d_A^-(U))$ for all $U\in \mathcal{C}$. It therefore follows from \Cref{NW} that there exists a $k$-arc-connected flip $J$ such that $f(U)\geq d_J^+(U)-d_J^-(U)$ for all $U\in \mathcal{C}$. In other words, for every vertex $u\in V$, \begin{equation}\label{eq:eulerian-1}
d_J^+(u)-d_J^-(u)\leq \left\lceil \frac{1}{2} \left(d_A^+(u)-d_A^-(u)\right)\right\rceil \end{equation} and
\begin{equation*}
d_J^+(V\setminus u)-d_J^-(V\setminus u)\leq
\left\lceil \frac{1}{2} \left(d_A^+(V\setminus u)-d_A^-(V\setminus u)\right)\right\rceil.
\end{equation*} The latter can be rewritten as \begin{equation*}
d_J^-(u)-d_J^+(u)\leq
\left\lceil \frac{1}{2} \left(d_A^-(u)-d_A^+(u)\right)\right\rceil;
\end{equation*} negating both sides, and reversing the inequality, we thus obtain \begin{equation}\label{eq:eulerian-2}
d_J^+(u)-d_J^-(u)
\geq
\left\lfloor \frac{1}{2} \left(d_A^+(u)-d_A^-(u)\right)\right\rfloor.
\end{equation}
Since
\eqref{eq:eulerian-1} and \eqref{eq:eulerian-2} hold for every $u\in V$, it follows that the digraph obtained after flipping the arcs in $J$ is near-Eulerian, as required.
\end{proof}

In fact, even in the conclusion of \Cref{NW} one can guarantee that after flipping the arcs in $J$ the digraph becomes near-Eulerian. This is obtained by updating $\mathcal{C}:=\mathcal{C}\cup \{\{u\},V\setminus u:u\in V\}$ and $f(U):=\left\lceil \frac{1}{2} (d_A^+(U)-d_A^-(U))\right\rceil$ for all $U\in \{\{u\},V\setminus u:u\in V\}$, and then applying \Cref{NW} to the updated crossing family and crossing submodular function.

%---------------------------------------------------------------------->
\subsection{$k$-arc-connected flips and $k$-dijoins}
%---------------------------------------------------------------------->

Before discussing the next set of applications of \Cref{main-1}, we need to set up the scene. Let $D=(V,A)$ be a digraph. A \emph{dicut} is an arc subset of the form $\delta^+(U)$ where $\delta^-(U)=\emptyset$, for some $U\subsetneq V, U\neq \emptyset$. A \emph{dijoin} is a subset $J\subseteq A$ that intersects every dicut at least once. Equivalently, $J$ is a dijoin if bidirecting every arc in $J$ makes the digraph $D$ strongly connected. In contrast, $J$ is a $1$-arc-connected flip if flipping every arc in $J$ makes the digraph strongly connected. Thus, every $1$-arc-connected flip is also a dijoin. It can be readily checked that the converse is not necessarily true. Interestingly, however, every inclusionwise minimal dijoin is a $1$-arc-connected flip (see~\cite{Schrijver03}, Theorem~55.1).
For an integer $k\geq 1$, a \emph{$k$-dijoin} is an arc subset that intersects every dicut at least $k$ times. Observe that the union of every $k$ disjoint dijoins is a $k$-dijoin (the converse is not necessarily true, see \Cref{fig:schrijver}). Moreover, we have the following important observation.

\begin{RE}\label{ACF->dijoin}
Given a digraph and an integer $k\geq 1$, every $k$-arc-connected flip is a $k$-dijoin.
\end{RE}

(The converse of this remark is not necessarily true even if the $k$-dijoin is inclusionwise minimal, see \Cref{fig:schrijver}).

\begin{figure}[ht]
\centering
\includegraphics[scale=0.3]{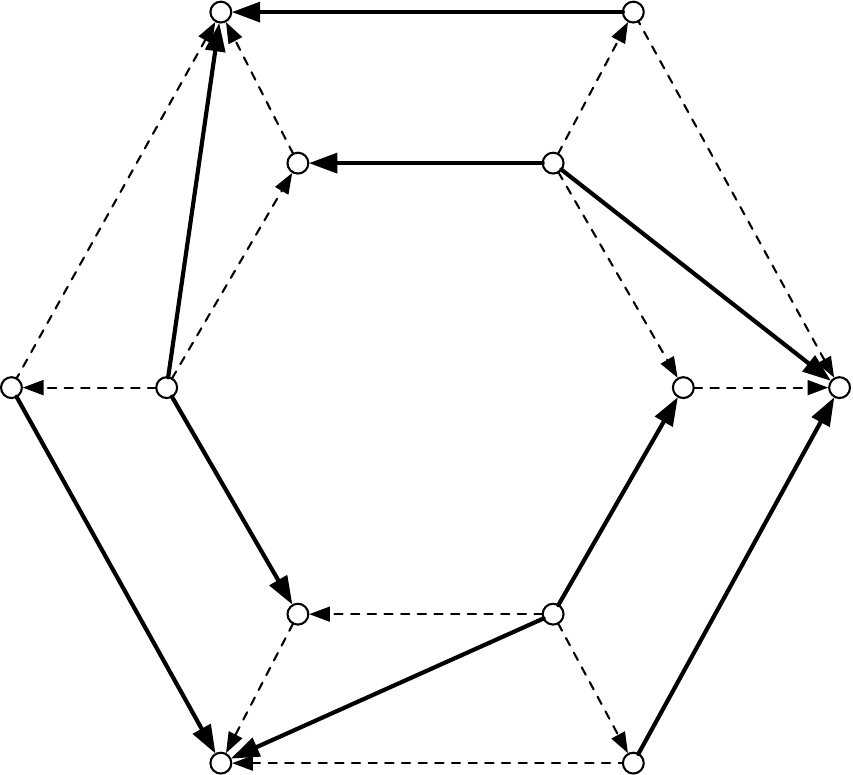}
\caption{The solid arcs form an inclusionwise minimal $2$-dijoin that cannot be decomposed into two dijoins~\cite{Schrijver80}, nor is it a $2$-arc-connected flip.}
\label{fig:schrijver}
\end{figure}

%---------------------------------------------------------------------->
\subsection{Woodall's conjecture and a weaker variant}
%---------------------------------------------------------------------->

A seminal result of Lucchesi and Younger is that the minimum size of a dijoin is equal to the maximum number of pairwise disjoint dicuts~\cite{Lucchesi78}. Woodall conjectured that the dual minimax relation also holds: the minimum size of a dicut is equal to the maximum number of pairwise disjoint dijoins~\cite{Woodall78}; this conjecture remains open. As a step towards the conjecture, it was recently shown that if the minimum size of a dicut is $\tau$, then $A$ may be decomposed into a dijoin and a $(\tau-1)$-dijoin~\cite{Abdi22+}. In fact, if Woodall's conjecture is true, then one should be able to decompose $A$ into a $k$-dijoin and a $(\tau-k)$-dijoin, for every integer $k\in \{1,\ldots,\tau-1\}$, but surprisingly even this remains open for $k\neq 1,\tau-1$. This leads us to the following weaker conjecture.

\begin{CN}\label{disjoint-dijoins-CN}
Let $\tau\geq 2$ be an integer.
Let $D=(V,A)$ be a digraph where every dicut has size at least $\tau$. Then $A$ can be decomposed into a $k$-dijoin and a $(\tau-k)$-dijoin, for every $k\in \{1,\ldots,\tau-1\}$.
\end{CN}

\Cref{main-1} has the following consequence that relates to \Cref{disjoint-dijoins-CN}.

\begin{theorem}\label{main-2}
Let $\tau,k$ be integers such that $\tau-1\geq k\geq 1$. Let $D=(V,A)$ be a digraph where $d_A^+(U)+(\frac{\tau}{k}-1)d_A^-(U)\geq \tau$ for all $U\subsetneq V, U\neq \emptyset$. Then $A$ can be decomposed into a $k$-arc-connected flip and a $(\tau-k)$-dijoin.
\end{theorem}
\begin{proof}
Let $\mathcal{C}$ be the family of subsets $U\subseteq V$ such that $\delta^+(U)$ is a dicut. Then $\mathcal{C}$ is a crossing family. Let $f:\mathcal{C}\to \mathbb{Z}$ be the function defined as $f(U):=d_A^+(U)-(\tau-k)$ for all $U\in \mathcal{C}$. Then $f$ is a crossing submodular (in fact, modular) function.
The inequalities
$d_A^+(U)+(\frac{\tau}{k}-1)d_A^-(U)\geq \tau$ for all $U\subsetneq V, U\neq \emptyset$, imply that $d_A^+(U)\geq \tau$ for all $U\in \mathcal{C}$, which in turn imply that $$f(U)\geq \frac{k}{\tau}d_A^+(U) =  \frac{k}{\tau}(d_A^+(U)-d_A^-(U)) \quad \forall U\in \mathcal{C};$$ the first inequality holds because $f(U)-\frac{k}{\tau}d_A^+(U) = \frac{\tau-k}{\tau}(d_A^+(U)-\tau)\geq 0$ for all $U\in \mathcal{C}$. We can now apply \Cref{main-1} to get a $k$-arc-connected flip $J$ such that $d_J^+(U)-d_J^-(U)\leq f(U)$ for all $U\in \mathcal{C}$. That is, for every dicut $\delta^+(U)$, we have $d_J^+(U)\leq d_A^+(U)-(\tau-k)$ which can be rewritten as $d_{A-J}^+(U)\geq \tau-k$. Thus, $A-J$ is a $(\tau-k)$-dijoin, implying that $(J,A-J)$ is the desired decomposition.
\end{proof}

Observe that the complement of a $k$-arc-connected flip is also a $k$-arc-connected flip, so for $\tau=2k$, \Cref{main-2} reduces simply to the weak orientation theorem, so our theorem extends this classic result in a different direction than \Cref{NW}.

%\paragraph{Special cases of \Cref{main-2}.} The inequalities on the cuts of $D$ in \Cref{main-2} are imposed for the simple reason that they are needed for the proof. That said, they possess some nice properties. For instance, the inequalities are equivalent to asking that $y=\frac{k}{\tau}\1$ satisfies $y(\delta^+(U)) -y(\delta^-(U))\geq k-d_A^-(U)$ for all $U\subsetneq V$, $U\neq \emptyset$. As $k$ increases, the inequalities become more strict: for $k=1$ the inequalities ask precisely that every dicut has size at least $\tau$ so we obtain \Cref{origin} below (see the proof for explanation), while for $k=\lfloor \tau/2\rfloor$ the inequalities \emph{almost} ask that every cut has size at least~$\tau$, in turn yielding \Cref{disjoint-dijoins} also below.

Let us discuss two special cases of \Cref{main-2}. In the special case $k=1$, we obtain the following.

\begin{theorem}[\cite{Abdi22+}]\label{origin}
Let $\tau\geq 2$ be an integer.
Let $D=(V,A)$ be a digraph where every dicut has size at least $\tau$. Then $A$ can be decomposed into a dijoin $J$ and a $(\tau-1)$-dijoin $J'$.
\end{theorem}
\begin{proof}
We claim that $d_A^+(U)+(\tau-1)d_A^-(U)\geq \tau$ for all $U\subsetneq V,U\neq \emptyset$. This holds because either $\delta^+(U),\delta^-(U)\neq \emptyset$ or one of $\delta^+(U),\delta^-(U)$ is a dicut. In the former case, $d_A^+(U)+(\tau-1)d_A^-(U)\geq 1+(\tau-1)= \tau$, and in the latter case, the dicut must have size at least $\tau$ by assumption, so the claimed inequality holds. It therefore follows from \Cref{main-2} for $k=1$ that $A$ can be decomposed into a $1$-arc-connected flip, which necessarily is a dijoin by \Cref{ACF->dijoin}, and a $(\tau-1)$-dijoin, as required.
\end{proof}

This theorem was proved recently in an attempt to prove Woodall's conjecture by first reducing the problem to a special class of \emph{sink-regular $(\tau,\tau+1)$-bipartite digraphs}~\cite{Abdi22+}. The proof we have given here bypasses this reduction.

The reader may wonder why this theorem does not automatically prove Woodall's conjecture, as one may try to repeat the argument on the subdigraph $D\setminus J$. A key complication comes from the fact that deleting an arc from $D$ may create a new dicut, whose size may unfavourably be smaller than $\tau-1$.
Another comes from the fact that given the decomposition $J\cup J'$, one may not necessarily be able to further decompose $J'$ into $\tau-1$ dijoins~\cite{Abdi22+}.

The second special case of \Cref{main-2} we consider is the case $k=\lfloor \tau/2\rfloor$.

\begin{theorem}\label{disjoint-dijoins}
Let $\tau\geq 2$ be an integer. Let $D=(V,A)$ be a digraph where every dicut has size at least~$\tau$. Suppose further that every cut of $D$ has size at least $\tau-1$, i.e., $|\delta^+(U)|+|\delta^-(U)|\geq \tau-1$ for all $U\subsetneq V,U\neq \emptyset$, and if equality holds, then the number of outgoing arcs is equal to the number of incoming arcs. Then $A$ can be decomposed into a $k$-arc-connected flip and a $(\tau-k)$-dijoin, for every $k\in \{1,\ldots,\lfloor \tau/2 \rfloor\}$.
\end{theorem}
\begin{proof}
Let $k\in \{1,\ldots,\lfloor \tau/2 \rfloor\}$. We claim that $d_A^+(U)+(\frac{\tau}{k}-1)d_A^-(U)\geq \tau$ for all $U\subsetneq V,U\neq \emptyset$. We know that $d_A^+(U)+d_A^-(U)\geq \tau-1$. If equality holds, then by assumption, $d_A^+(U)=d_A^-(U)=\frac{\tau-1}{2}$, so $d_A^+(U)+(\frac{\tau}{k}-1)d_A^-(U) = \frac{\tau-1}{2}\cdot \frac{\tau}{k}\geq \tau$. Otherwise, $d_A^+(U)+d_A^-(U)\geq \tau$, so $d_A^+(U)+(\frac{\tau}{k}-1)d_A^-(U)\geq
d_A^+(U)+d_A^-(U)\geq \tau$, proving the claimed inequality. The theorem now follows from \Cref{main-2}.
\end{proof}

In particular, this proves \Cref{disjoint-dijoins-CN} when the underlying undirected graph is $\tau$-edge-connected:

\begin{theorem}\label{disjoint-dijoins-2}
Let $\tau\geq 2$ be an integer.
If $D=(V,A)$ is a digraph whose underlying undirected graph is $\tau$-edge-connected, then $A$ can be decomposed into a $k$-dijoin and a $(\tau-k)$-dijoin, for every $k\in [\tau-1]$.
\end{theorem}
\begin{proof}
By symmetry we may assume that $k\leq \tau-k$, so $k\in \{1,\ldots,\lfloor \tau/2 \rfloor\}$. Thus, since every cut of $D$ has size at least $\tau$, we may apply \Cref{disjoint-dijoins} to decompose $A$ into a $k$-arc-connected flip, which necessarily is a $k$-dijoin by \Cref{ACF->dijoin}, and a $(\tau-k)$-dijoin.
\end{proof}

\Cref{disjoint-dijoins-2} suggests that it may be easier to prove Woodall's conjecture for $\tau$-edge-connected instances. After all, if the underlying undirected graph has $\tau$ disjoint spanning trees, which is guaranteed by $2\tau$-edge-connectivity for instance~\cite{Nash-Williams61,Tutte61}, then the digraph has $\tau$ disjoint dijoins. It should be noted that Woodall's conjecture for $\tau=3$ has been proven for $4$-edge-connected instances, that is, if the underlying undirected graph is $4$-edge-connected, then the digraph contains 3 disjoint dijoins~\cite{Meszaros18}.

%---------------------------------------------------------------------->
\subsection{Packing dijoins in weighted digraphs}
%---------------------------------------------------------------------->

Finally, \Cref{main-1} leads to an intriguing extension of \Cref{main-2} to a setting where arcs are assigned nonnegative integer weights, viewed as capacities for packing dijoins. By replacing an arc of weight $t\geq 1$ with $t$ parallel arcs of weight $1$, we may reduce to $0,1$ weights, so we can focus solely on them. Given a digraph $D=(V,A)$ and $J\subseteq A$, denote by $D[J]$ the subdigraph with vertex set $V$ and arc set~$J$.

\begin{theorem}\label{main-3}
Let $\tau,k$ be integers such that $\tau-1\geq k\geq 1$.
Let $D=(V,A)$ be a digraph, and $w\in \{0,1\}^A$. Suppose $w(\delta^+(U))+(\frac{\tau}{k}-1)w(\delta^-(U))\geq \tau$ for all $U\subsetneq V$, $U\neq \emptyset$. Then $\{a\in A:w_a=1\}$ can be decomposed into a $k$-arc-connected flip of $D[\{a\in A:w_a=1\}]$ and a $(\tau-k)$-dijoin of $D$.
\end{theorem}
\begin{proof}
The proof is similar to that of \Cref{main-2}. This time, however, we apply \Cref{main-1} to the digraph $D[\{a\in A:w_a=1\}]$, denoted as $D'=(V,A')$, the crossing family $\mathcal{C}:=\{U\subsetneq V: \delta_D^-(U)=\emptyset,U\neq \emptyset\}$, and the crossing submodular function $f:\mathcal{C}\to \mathbb{Z}$ defined as $f(U)=d_{A'}^+(U)-(\tau-k)=w(\delta_D^+(U))-(\tau-k)$. We include the proof for completeness.

The inequalities
$w(\delta_D^+(U))+(\frac{\tau}{k}-1)w(\delta_D^-(U))\geq \tau$ for all $U\subsetneq V, U\neq \emptyset$, imply that $w(\delta_D^+(U))\geq \tau$ for all $U\in \mathcal{C}$, which in turn imply that $$f(U)\geq \frac{k}{\tau}w(\delta_D^+(U)) =  \frac{k}{\tau}(w(\delta_D^+(U))-w(\delta_D^-(U)))=\frac{k}{\tau}(d_{A'}^+(U)-d_{A'}^-(U)) \quad \forall U\in \mathcal{C}.$$
We may therefore apply \Cref{main-1} to get a $k$-arc-connected flip $J\subseteq A'$ of $D'$ such that $d_J^+(U)-d_J^-(U)\leq f(U)$ for all $U\in \mathcal{C}$. That is, for every dicut $\delta_D^+(U)$ of $D$, $d_J^+(U)\leq d_{A'}^+(U)-(\tau-k)$ which can be rewritten as $
d_{A'-J}^+(U)\geq \tau-k$. Thus, $A'-J$ is a $(\tau-k)$-dijoin of $D$, implying that $(J,A'-J)$ is the desired decomposition.
\end{proof}

By specializing \Cref{main-3} to $k=1$ we obtain the following.

\begin{theorem}\label{main-3-CO}
Let $D=(V,A)$ be a digraph, let $\tau\geq 2$ be an integer, and $w\in \{0,1\}^A$.
Suppose $\min\{w(\delta^+(U)),w(\delta^-(U))\}\geq 1$ or $\max\{w(\delta^+(U)),w(\delta^-(U))\}\geq \tau$ for all $U\subsetneq V$, $U\neq \emptyset$. Then $\{a\in A:w_a=1\}$ can be decomposed into a dijoin and a $(\tau-1)$-dijoin of $D$.\qed
\end{theorem}

%%%%%%%%%%%%%%%%%%%%%%%%%%%%%%%%%%%%%%%
%%%%%%%%%%%%%%%%%%%%%%%%%%%%%%%%%%%%%%%
%%%%%%%%%%%%%%%%%%%%%%%%%%%%%%%%%%%%%%%
\section{Discussion on computational complexity}\label{sec:complexity}
%%%%%%%%%%%%%%%%%%%%%%%%%%%%%%%%%%%%%%%
%%%%%%%%%%%%%%%%%%%%%%%%%%%%%%%%%%%%%%%
%%%%%%%%%%%%%%%%%%%%%%%%%%%%%%%%%%%%%%%

To discuss the computational complexity of \Cref{main-1} and \Cref{main-0}, we need to make some standard assumptions about how the crossing families and the crossing submodular functions are provided. Let us lay the ground work.

A \emph{lattice family} $\mathcal{L}$ over a finite ground set~$V$ is one where for all $U,W\in \mathcal{L}$ we have $U\cap W,U\cup W\in \mathcal{L}$. Then the lattice family $\mathcal{L}$ can be provided compactly as follows. Let $L,M$ be the inclusionwise minimal and maximal sets in $\mathcal{L}$, respectively. Define the relation $\preceq$ on $V$ as follows: $u\preceq v$ if every set in $\mathcal{L}$ that contains $v$ also contains $u$. It can be readily checked that $\preceq$ is a preorder, that is, it is reflexive and transitive. Furthermore, it can be readily checked that $U\in \mathcal{L}$ if and only if $L\subseteq U\subseteq M$ and $U$ is a lower ideal for $\preceq$ (that is, if $v\in U$ and $u\preceq v$, then $u\in U$). Subsequently, $\mathcal{L}$ is fully characterized by $L,M$ and $\preceq$, implying in turn that every lattice family can be described compactly~(see \S49.3 of \cite{Schrijver03} for more). We say that the lattice family $\mathcal{L}$ is \emph{well-provided} if it is described by $L,M$ and $\preceq$.

Let $\mathcal{C}$ be a crossing family over ground set $V$. It can be readily seen that for every pair of elements $u,v\in V$, $\mathcal{C}_{uv}:=\{C\in \mathcal{C}:u\in C,v\notin C\}$ is a lattice family~(see \S49.10 of \cite{Schrijver03} for more). Clearly, $\mathcal{C}$ can be fully described by all these lattice families. We say that the crossing family $\mathcal{C}$ is \emph{well-provided} if all the lattice families $\mathcal{C}_{uv}, u,v\in V,u\neq v$ are well-provided.

Given a well-provided crossing family $\mathcal{C}$, and a crossing submodular function $f:\mathcal{C}\to \mathbb{Z}$, a \emph{value oracle} for $f$ is an oracle which, given $X\in \mathcal{C}$, outputs in unit time the value of $f(X)$. We are ready to state a preliminary result.

\begin{theorem}\label{thm:strong-poly}
There exists an algorithm that, given well-provided crossing families $\mathcal{C}_i,i=1,2$ over ground set $V$, crossing submodular functions $f_i:\mathcal{C}_i\to \mathbb{Z},i=1,2$ provided via value oracles, an integer $k$, and $w\in \mathbb{Z}^V$, outputs an extreme optimal solution of $\max\{w^\top x \,:\, x(V)= k,\, x(U)\leq f_i(U)~\forall U\in \mathcal{C}_i,\, i=1,2\}$ in oracle strongly polynomial time.
\end{theorem}

Before we provide a proof sketch of this theorem, let us stress that it is well-known that $\max\{w^\top x\,:\, x(V)= k,\, x(U)\leq f_i(U)~\forall U\in \mathcal{C}_i,\, i=1,2\}$ is strongly polynomial time solvable; see for instance \S47.4 of \cite{Schrijver03}. Though the statement above is widely accepted, none of the algorithms we could find in the literature necessarily return an extreme optimal solution. However, this issue can be addressed by finding a lexicographically maximal optimal solution.

\begin{proof}[Proof sketch.]
Let $\widetilde{P}:=\{x\in \mathbb{R}^V \,:\, x(V)= k,\, x(U)\leq f_i(U)~\forall U\in \mathcal{C}_i,i=1,2\}$, and denote by $\widetilde{F}$ the optimal face of $\max\{w^\top x \,:\, x\in \widetilde{P}\}$. It follows from \S49.3, \S49.7, and \S49.10 of \cite{Schrijver03} that one can construct, in strongly polynomial time, submodular functions $f'_i$, $i=1,2$ defined over $2^V$ such that $\widetilde{F}$ is the optimal face of $\max\{w^\top x \,:\, x(V)=f'_i(V),\, x(U)\le f'_i(U)~~\forall~U\subset V,\, i=1,2 \}$.
The proof of Theorem 47.4 of  \cite{Schrijver03} shows how to construct, in strongly polynomial time, submodular functions $f''_i$, $i=1,2$ defined over $2^V$ such that
$$\widetilde{F}=\left\{x\in \mathbb{R}^V\,:\, x(V)=f''_i(V),\, x(U)\le f''_i(U)~~\forall~U\subset V,\, i=1,2\right\}.$$
This shows that the original problem can be reduced, in strongly polynomial time, to the problem of finding an extreme common base of two extended polymatroids. While there exist strongly polynomial time algorithms that can provide a common base~\cite{Lawler82} (see Theorem 47.1 in  \cite{Schrijver03}), the common base returned by these algorithms may not be extreme. However, for any given $v\in V$, one can in strongly polynomial time find a point in $\widetilde{F}$ maximizing $x_v$~\cite{Frank84a} (see Theorem 47.2 of \cite{Schrijver03}). This implies, as we discuss next, that we can pick an ordering $1,2,\ldots,n$ of the elements in $V$ and compute a lexicographically maximal point in $\widetilde{F}$, which will therefore be extreme. Indeed, we start by computing the value $\alpha:=\max\{x_n: x \in \widetilde{F}\}$. To iterate, we need to compute the largest value of $x_{n-1}$ for $x\in \widetilde{F}\cap \{x : x_n=\alpha\}$. This can be reduced again to the problem of finding a common base maximizing $x_{n-1}$ for two extended polymatroids defined over $\{1,\ldots,n-1\}$, where the functions $g_i, i=1,2$ are defined over $2^{\{1,\ldots,n-1\}}$ by $g_i(U)=\min\{f''_i(U), f''_i(U\cup \{n\})-\alpha\}$ for all $U\subseteq V\setminus \{n\}$. It can be readily checked that $g_i, i=1,2$ are indeed still submodular, and furthermore, a point $z\in \mathbb{R}^{\{1,\ldots,n-1\}}$ is a common base of $g_1$ and $g_2$ if and only if $(z,\alpha)$ is a common base of $f''_1$ and $f''_2$. Hence we are able to apply Theorem 47.2 of \cite{Schrijver03} iteratively, repeating the process for $1,\ldots,n-2$, and so on and so forth.
\end{proof}

\paragraph{Complexity aspects of \Cref{main-0}.} To discuss the complexity aspects of part (b) of this theorem, we assume that for $i=1,2$, the crossing family $\mathcal{C}_i$ is well-provided, and $f_i$ is provided via a value oracle. Furthermore, the non-empty face of $P$ is provided by some vector $c\in \mathbb{Z}^A$, namely, the face $F$ is the optimal face of $\max\{c^\top y:y\in P\}$.

The first question is whether the cut condition \eqref{eq:cut-condition} can be verified in polynomial time. We do not know. The verification amounts to determining whether $g(U):=u(\delta^+(U))-\ell(\delta^-(U))-\min_{i=1,2} f_i(U)$ $\forall U\subsetneq V,\, U\neq \emptyset$ is a nonnegative function. However, $g$ is not necessarily a crossing submodular function, making the verification problem challenging. That said, there is a way to circumvent this issue altogether.

Following the proof of \Cref{main-0}~(b), we can show that in strongly polynomial time, we can either find an integral point $y^\star\in F$ such that $\ell\leq y^\star\leq u$, or a subset $U\subseteq V$ that violates the cut condition~\eqref{eq:cut-condition}.

%\begin{RE}
%\label{rem:strong_poly} Let be $P$ defined as in \eqref{eq:sub_flow_intersection}, $\ell\in (\mathbb{Z}\cup\{-\infty\})^A$, $u\in (\mathbb{Z}\cup\{+\infty\})^A$ satisfying $\ell\leq u$, and let $F$ be a non-empty face of $P$ provided as the optimal face of $\max\{c^\top y:y\in P\}$, for a given $c\in \mathbb{Z}^A$. In strongly polynomial time we can either find an integral point $y^\star\in F$ such that $\ell\leq y^\star\leq u$ or a $U\subseteq V$ violating the cut condition~\eqref{eq:cut-condition}.
%\end{RE}
First we compute $w\in \mathbb{Z}^V$ such that $w^\top M=c^\top$, which can be done in strongly polynomial time. Then, if we let $\widetilde{P}$ and $\widetilde{F}$ be defined as in the proof of \Cref{main-0}, we have that $\widetilde{F}$ is the optimal face of $\max\{w^\top x:x\in \widetilde{P}\}$. Next, we find a vertex $b$ of $\widetilde{F}$ in strongly polynomial time, by \Cref{thm:strong-poly}. Note that $b$ ought to be integral by \Cref{box-TDI}. Then we find either an integral $b$-transhipment $y^\star$ satisfying $\ell \le y^\star\le u$, or a certificate that such a $b$-transhipment does not exist, in the form of a set $U$, $U\subsetneq V,U\neq \emptyset$, such that $b(U)> u(\delta^+(U))-\ell(\delta^-(U))$; this can be done in strongly polynomial time (see Corollary 12.2d of \cite{Schrijver03}). In the first case, we get an integral point $y^\star$ in $F\cap\{y\in\mathbb{R}^A : \ell \le y\le u\}$, and in the second case, we have a certificate $U$ that \eqref{eq:cut-condition} is not satisfied.

%}
\paragraph{Complexity aspects of \Cref{main-1}.}
First, we claim that the assumed inequalities $d_A^+(U)+(\frac{\tau}{k}-1)d_A^-(U)\geq \tau$ for all $U\subsetneq V, U\neq \emptyset$, can be verified in strongly polynomial time.
To elaborate, note that the inequalities are satisfied if, and only if, $\bar{y}=\frac{k}{\tau}\cdot \1$ satisfies $y(\delta^+(U))-y(\delta^-(U))\leq d^+_A(U)-k=f_2(U)$ for all $U\in \mathcal{C}_2=\{U\subsetneq V: U\neq \emptyset\}$. This holds if, and only if, $\min\{g(U):U\in \mathcal{C}_2\}\geq 0$ for the crossing submodular function $g$ over the crossing family $\mathcal{C}_2$ defined as $g(U):=f_2(U) - \bar{y}(\delta^+(U))+\bar{y}(\delta^-(U))$. The claim now follows from the fact that the minimization problem can be solved in strongly polynomial time~(see \S49.10 of \cite{Schrijver03} for details).

Secondly, \Cref{main-1} also assumes the inequalities $f(U)\geq \frac{k}{\tau}(d_A^+(U)-d_A^-(U))$ for all $U\in \mathcal{C}$. To be able to verify these inequalities in polynomial time, we would need some assumptions on how the data is provided. As such, suppose $\mathcal{C}$ is a well-provided crossing family, and $f$ is provided via a value oracle. Then, just as above, the inequalities can be verified in strongly polynomial time~(see \S49.10 of \cite{Schrijver03}).

Finally, with the assumptions mentioned above on $\mathcal{C}=\mathcal{C}_1$ and $f=f_1$, we can find the $k$-arc-connected flip $J$ in strongly polynomial time, by an application of the algorithm we provided for \Cref{main-0}~(b). To describe the algorithm, first we find $b\in \mathbb{Z}^V$ such that $\1^\top b=0$, $b(U)\leq f_i(U)$ for all $U\in \mathcal{C}_i$, then we find a $b$-transshipment $y^\star\in \{0,1\}^A$. The first step relies on the fact that every vertex of $\{x\,:\, \1^\top x=0,\, x(U)\leq f_i(U)~~\forall~U\in \mathcal{C}_i,\, i=1,2\}$ is integral by \Cref{box-TDI}, and that a vertex $b$ can be found in strongly polynomial time by \Cref{thm:strong-poly}. The second step can also be done in strongly polynomial time~(see Corollary 12.2d of \cite{Schrijver03}). The subset $J$ corresponds to the support of~$y^\star$.

%%%%%%%%%%%%%%%%%%%%%%%%%%%%%%%%%%%%%%%
%%%%%%%%%%%%%%%%%%%%%%%%%%%%%%%%%%%%%%%
%%%%%%%%%%%%%%%%%%%%%%%%%%%%%%%%%%%%%%%
\section{Open questions and concluding remarks}\label{sec:conclusion}
%%%%%%%%%%%%%%%%%%%%%%%%%%%%%%%%%%%%%%%
%%%%%%%%%%%%%%%%%%%%%%%%%%%%%%%%%%%%%%%
%%%%%%%%%%%%%%%%%%%%%%%%%%%%%%%%%%%%%%%

\Cref{main-3-CO} connects intriguingly to a conjecture of Chudnovsky, Edwards, Kim, Scott, and Seymour~\cite{Chudnovsky16} for $0,1$-weighted digraphs $(D,w)$, that if every cut has nonzero weight and every dicut has weight at least $\tau$, then $\{a\in A:w_a=1\}$ can be decomposed into $\tau$ dijoins of $D$. In fact, another result of ours, namely \Cref{thm:integral_intersection_weakly_connected}, has a direct application to this conjecture for $\tau=2$. In this case, the conjecture can be equivalently formulated as below.

\begin{CN}[\cite{Chudnovsky16}]\label{spanning-tree-CN}
Let $G=(V,E)$ be a tree, and let $\mathcal{L}$ be a lattice family over ground set $V$ such that $|\delta(U)|\geq 2$ for all $U\in \mathcal{L}\setminus \{\emptyset,V\}$. Then there exists an orientation $D$ of $G$ such that $\delta_D^+(U),\delta_D^-(U)\neq \emptyset$ for all $U\in \mathcal{L}\setminus \{\emptyset,V\}$.
\end{CN}

To point out the connection to \Cref{thm:integral_intersection_weakly_connected}, let $D=(V,A)$ be an arbitrary orientation of $G$. Consider the system \begin{align*}
y(\delta_D^+(U))-y(\delta_D^-(U))&\leq d_A^+(U)-1 \quad \forall U\in \mathcal{L}\setminus \{\emptyset,V\}\\
 y(\delta_D^+(U))-y(\delta_D^-(U))&\geq 1-d_A^-(U) \quad \forall U\in \mathcal{L}\setminus \{\emptyset,V\}.
\end{align*} Observe that $y=\frac12 \cdot \1$ is a feasible solution. Observe further that $\mathcal{L}\setminus \{\emptyset,V\}$ is a crossing family. Since $D$ is weakly connected, it follows from \Cref{thm:integral_intersection_weakly_connected} that the system is TDI, and so it has an integral solution. \Cref{spanning-tree-CN} states equivalently that the system has a $0,1$ solution.

The next open question comes from \Cref{thm:integral_intersection_weakly_connected}. We saw that the weak connectivity assumption could not be dropped due to \Cref{bad-example}. However, this example is not only disconnected but has three connected components in its underlying undirected graph. An immediate open question is whether \Cref{thm:integral_intersection_weakly_connected} extends beyond weakly connected digraphs to those digraphs with at most two connected components in their underlying undirected graph?

Another open question comes from a closer look at \Cref{main-2}. The theorem provides a sufficient condition for decomposing the arc set of a digraph $D$ into a $k$-arc-connected flip and a $(\tau-k)$-dijoin. Clearly, to be able to do this, two conditions are necessary: (a) the underlying undirected graph of $D$ must be $2k$-edge-connected, and (b) every dicut of $D$ must have size at least $\tau$. Are these two conditions also sufficient?

%\begin{CN}\label{main-con}
%Let $\tau,k$ be integers such that $\tau-1\geq k\geq 1$, and let $D=(V,A)$ be a digraph. Then $A$ can be decomposed into a $k$-arc-connected flip and a $(\tau-k)$-dijoin if, and only if, every cut has size at least $2k$ and every dicut has size at least $\tau$.
%\end{CN}

Finally, as we show in \S\ref{sec:fractional} of the appendix, the intersection of two submodular flow systems with integral right-hand sides, is not necessarily box-integral. In March 2023, during a talk at the Combinatorics and Optimization workshop at ICERM, Brown University, the first author made the following ``wild" conjecture:
\emph{
Let $D=(V,A)$ be a weakly connected digraph and, for $i=1,2$, let $\mathcal{C}_i$ be a crossing family over ground set $V$ and $f_i:\mathcal{C}_i\to \mathbb{Z}$ be a crossing submodular function. Then the polyhedron $P$ from \eqref{eq:sub_flow_intersection} is box-half-integral.
}
However, this conjecture has been refuted by Goemans and Pan~\cite{Goemans23+}.

%%%%%%%%%%%%%%%%%%%%%%%%%%%%%%%%%%%%%%%
%%%%%%%%%%%%%%%%%%%%%%%%%%%%%%%%%%%%%%%
%%%%%%%%%%%%%%%%%%%%%%%%%%%%%%%%%%%%%%%
\section*{Acknowledgements}
%%%%%%%%%%%%%%%%%%%%%%%%%%%%%%%%%%%%%%%
%%%%%%%%%%%%%%%%%%%%%%%%%%%%%%%%%%%%%%%
%%%%%%%%%%%%%%%%%%%%%%%%%%%%%%%%%%%%%%%

We would like to thank Krist\'{o}f B\'{e}rczi, Tam\'{a}s Kir\'{a}ly, and Zolt\'{a}n Szigeti for the fruitful discussions regarding the content of this paper. Special thanks to three anonymous referees whose detailed comments vastly improved the presentation of this paper. This work was supported in part by ONR grant N00014-22-1-2528, and EPSRC grant EP/X030989/1.

{\small \bibliographystyle{abbrv}\bibliography{references}}

\appendix

\section{Hardness for three base systems, and two submodular flow systems}\label{sec:two-subs}

\paragraph{Integral solutions to three base systems.} Take three matroids over the same (finite) ground set $V$ with rank functions $r_1,r_2,r_3$, respectively, where the functions are given via an oracle (which for every $X\subseteq V$ outputs $r_i(X)$ in unit time). Suppose further $r_i(V)=r$ for all $i$, and that $r_1(V\setminus u)=r$ and $r_1(u)=1$ for all $u\in V$. It is known that finding a common basis of the three matroids is a hard problem in general. For example, it includes the NP-complete problems \emph{Does a bipartite graph have a Hamilton cycle?}~\cite{Akiyama80} and \emph{Does a digraph have a Hamilton $st$-dipath?}~(see \S3.1.3 of \cite{Garey79}).

Consider the intersection of the three base systems $
\1^\top x=r, x(U)\leq r_i(U)~\forall U\subseteq V, i=1,2,3$. Suppose $x\in \mathbb{Z}^V$ is an integral solution. It can be readily checked, by using the assumption that $r_1(u)=1$ and $r_1(V\setminus u)=r$ for all $u\in V$, that $x\in \{0,1\}^V$, and $\{u\in V:x_u=1\}$ is a common basis of $M_1,M_2,M_3$. Since finding a common basis of three matroids is hard, finding a solution to three base systems in $\mathbb{Z}^A$ is also hard.

\paragraph{Integral solutions to two submodular flow systems.} Let $D=(V,A)$ be a digraph, and let $f_i:\mathcal{C}_i\to \mathbb{Z}, i=1,2$ be two crossing submodular functions. Consider the system $y(\delta^+(U))-y(\delta^-(U))\leq f_i(U)~\forall U\in \mathcal{C}_i,i=1,2$. Finding an integral solution to this system is a hard problem.

To see this, let $M_i,i=1,2,3$ be the matroids as above, with rank functions $r_i,i=1,2,3$, respectively. Let $V^\star$ be a copy of $V$. Let $D$ be the digraph over vertex set $V\cup V^\star$, and arc set $A:=\{(u,u^\star):u\in V\}$. Let $\mathcal{C}:=\{U:U\subseteq V\}\cup \{V\cup \overline{U^\star}:U\subseteq V\}\cup \{V^\star\}$, where $U^\star\subseteq V^\star$ corresponds to the subset $U\subseteq V$, and $\overline{U^\star}=V^\star-U^\star$. It can be readily checked that $\mathcal{C}$ is a crossing family.

Define $f:\mathcal{C}\to \mathbb{Z}$ as follows: $f(U):=r_1(U)$ for all $U\subseteq V$, $f(V\cup \overline{U^\star}):=r_2(U)$ for all $U\subseteq V$, and $f(V^\star):=-r_1(V)$ (note $f(V)=r$ and $f(V^\star)=-r$). It can be readily checked that $f$ is a crossing submodular function over $\mathcal{C}$. Suppose $y\in \mathbb{Z}^A$ is an $f$-submodular flow in $D$. It can be readily checked, by using the assumption that $r_1(u)=1$ and $r_1(V\setminus u)=r$ for all $u\in V$, that $y\in \{0,1\}^A$, and $\{u\in V:y_{(u,u^\star)}=1\}$ is a common basis of $M_1,M_2$.
Conversely, for every common basis $B$ of $M_1,M_2$, the incidence vector of $\{(u,u^\star):u\in B\}$ is an $f$-submodular flow in $\{0,1\}^A$.
Subsequently, there exists an $f$-submodular flow in $\mathbb{Z}^A$ if and only if $M_1,M_2$ have a common basis.

Similarly, define $g:\mathcal{C}\to \mathbb{Z}$ as follows: $g(U):=r_1(U)$ for all $U\subseteq V$, $f(V\cup \overline{U^\star}):=r_3(U)$ for all $U\subseteq V$, and $g(V^\star):=-r_1(V)$. Then $g$, too, is a crossing submodular function over $\mathcal{C}$, and there exists a $g$-submodular flow in $\mathbb{Z}^A$ if and only if $M_1,M_3$ have a common basis.

Putting it altogether, we get that there exists a $y\in \mathbb{Z}^A$ that is both an $f$- and $g$-submodular flow if and only if $M_1,M_2,M_3$ have a common basis. Since finding a common basis of three matroids is hard, finding a solution to two sets of submodular flow constraints in $\mathbb{Z}^A$ is also hard.\footnote{Our argument can easily be adapted to show that finding an integral solution to the intersection of two submodular flow systems, contains the problem of finding a common basis of four matroids.}

\section{\Cref{main-0} and box constraints}\label{sec:fractional}

Here we demonstrate that the system \eqref{eq:sub_flow_intersection} from \Cref{main-0}, together with box constraints, is not necessarily integral. To this end, consider the digraph $D=(V,A)$ displayed in \Cref{fig:fractional} (left).

\begin{figure}[ht]
\centering
\includegraphics[scale=0.3]{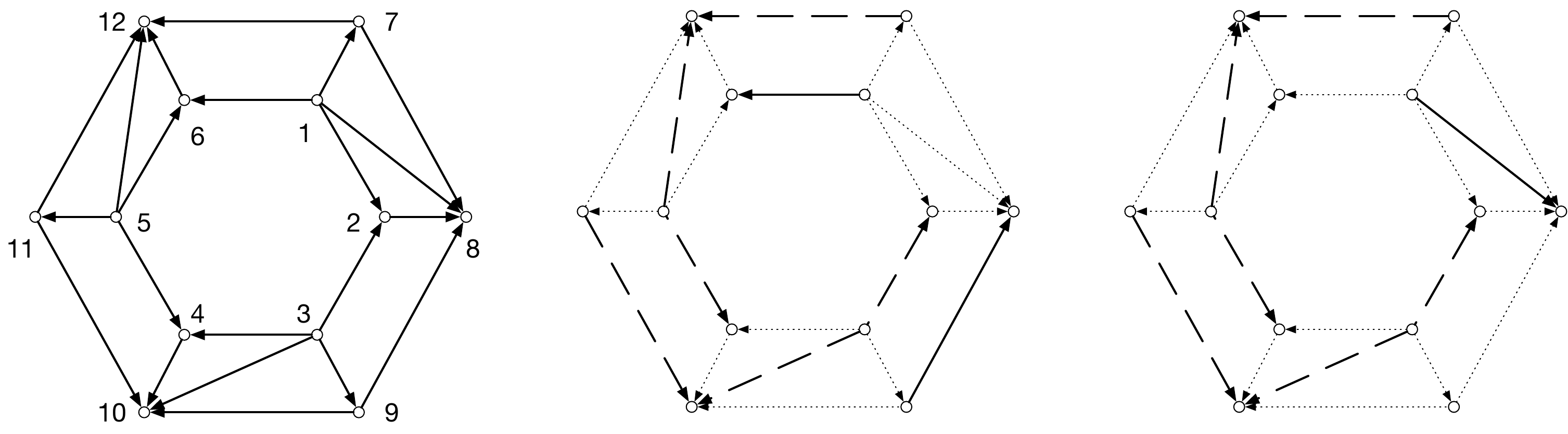}
\caption{
Left: A vertex-labelled digraph $D=(V,A)$. Middle/Right: Representations of two fractional vertices $y^1/y^2$ of the polytope  $Q=P\cap [\0, \1]$, where solid arcs are set to $1$, dotted arcs to $0$, and dashed arcs to $\frac12$.}
\label{fig:fractional}
\end{figure}

Define the crossing families $\mathcal{C}_1:=\{U\subsetneq V: \delta^-(U)=\emptyset, U\neq \emptyset\}$ and $\mathcal{C}_2:=\{U\subsetneq V:U\neq \emptyset\}$, and the crossing submodular functions $f_1(U):=d_A^+(U)-3~\forall U\in \mathcal{C}_1$ and $f_2(U):=d_A^+(U)-1~\forall U\in \mathcal{C}_2$. Clearly $f_2(U)\leq d_A^+(U)~\forall U\in \mathcal{C}_2$. Thus, $\ell=\0$ and $u=\1$ satisfy the cut condition~\eqref{eq:cut-condition}. Let $P$ be the polyhedron defined as in~\eqref{eq:sub_flow_intersection}, and  $Q:=P\cap [\0, \1]$.

To give some intuition, a $0,1$ vector $\bar{y}$ belongs to $Q$ if, and only if, $\{a\in A:\bar{y}_a=0\}$ is a $3$-dijoin and $\{a\in A:\bar{y}_a=1\}$ is a $1$-arc-connected flip of $D$. This gives a description of all the integral vertices of $Q$. However, $Q$ may have fractional vertices. To see this, define $y^1\in \left\{0,\frac12,1\right\}^A$ where for each $a\in A$, $y^1_a=0,\frac12$ or $1$ if $a$ is dotted, dashed, or solid in \Cref{fig:fractional} (middle), respectively. Define $y^2\in \left\{0,\frac12,1\right\}^A$ analogously with respect to \Cref{fig:fractional} (right).

\begin{PR}\label{pr:fractional}
$y^1,y^2$ are vertices of $Q$.
\end{PR}
\begin{proof}
It can be readily checked that $y^1,y^2$ are feasible. To see that $y^i$ is a vertex, we need to exhibit $21$ linearly independent tight inequalities at $y^i$. We immediately get $15$ from $\0\leq y\leq \1$ as $y^i$ has as many coordinates set to $0$ or $1$. For $y^1$ the remaining $6$ may be chosen as the following inequalities: $y(\delta^+(U))-y(\delta^-(U))\leq f_1(U)$ for $U=
\{3\},\overline{\{10\}},\{5\},\overline{\{12\}}, \{1,2,3,7,8,9\}$ and
$y(\delta^+(U))-y(\delta^-(U))\leq f_2(U)$ for $U=\{7,8,9,10,11,12\}$. For $y^2$ they may be chosen as $y(\delta^+(U))-y(\delta^-(U))\leq f_1(U)$ for $U=
\{3\},\overline{\{10\}},\{5\},\overline{\{12\}}, \{1,5,6,7,11,12\}$ and
$y(\delta^+(U))-y(\delta^-(U))\leq f_2(U)$ for $U=\{4,5,6,10,11,12\}$.
\end{proof}

\end{document}